\documentclass[11pt,a4paper,fullpage]{article}

\usepackage{amsmath,amssymb,amsthm}
\usepackage{bbm, bbold}
\usepackage[latin2]{inputenc}
\usepackage[T1]{fontenc}

\usepackage[font=it]{caption}

\usepackage{graphicx}

\usepackage[a4paper, total={6in, 8in}]{geometry}

\usepackage{bigints}

\usepackage{appendix}

\newcommand{\E }{\operatorname{E}}
\renewcommand{\P}{\operatorname{P}}

\newtheorem{thm}{Theorem}
\newtheorem{cor}[thm]{Corollary}
\newtheorem{rmk}[thm]{Remark}

\begin{document}
\title{\bf Asymptotics of the order statistics for a process with~a~regenerative structure}
\author{Natalia Soja-Kukie\l a
\thanks{Faculty of Mathematics and Computer Science, Nicolaus Copernicus University, ul.~Chopina 12/18, 87-100 Toru\'n, Poland; e-mail address: \texttt{natas@mat.umk.pl}}}
\date{}
\maketitle

\begin{abstract}
In the paper, a regenerative process $\{X_n:n\in\mathbb{N}\}$ with finite mean cycle length is considered. For~$M_n^{(q)}$ denoting the $q$-th largest value  in $\{X_k : 1\leqslant k \leqslant n\}$, we prove that
\begin{equation*}
\sup_{x\in\mathbb{R}} \left|P\left(M^{(q)}_n\leqslant x\right)  - G(x)^n \sum_{k=0}^{q-1}\frac{\left(-\log G(x)^n\right)^k}{k!}\gamma_{q,k}(x)\right| \to 0,\quad \text{as} \quad n\to\infty,
\end{equation*}
for  $G$ and $\gamma_{q,k}$ expressed in terms of maxima over the cycle. The result is illustrated with examples.

\vspace{0.3cm}


\vspace{0.3cm}

\end{abstract}

\section{Introduction}

Let $\{X_n:n\in\mathbb{N}\}$ be a sequence of random variables. For $j<n$, we denote by $M^{(q)}_{j,n}$ the~$q$-th largest value of $X_{j+1}, X_{j+2}, \ldots, X_n$ if $q\leqslant n-j$ and put $M^{(q)}_{j,n}:=-\infty$ otherwise. For~convenience, we write $M_n^{(q)}$ for $M_{0,n}^{(q)}$,  $M_{j,n}$ for $M^{(1)}_{j,n}$ and $M_n$ for $M^{(1)}_{0,n}$.

In the paper we investigate the asymptotic behaviour of $M_n^{(q)}$, as $n\to\infty$, for $\{X_n\}$ with a regenerative structure. The general theory of regenerative processes can be found in Asmussen \cite{AS}. Here, a minimal formalism corresponding to limit theorems for first $r$ order statistics is adopted. We say that $\{X_n\}$ has an \mbox{\emph{r-max-regenerative structure}} with $r\in\mathbb{N}_+$, if there exist integer-valued random variables $0 < S_0 <S_1<\ldots$ (we put $S_{-1}:=0$), representing regeneration times, such that for cycles 
$ C_n:= \{X_k: S_{n-1} \leqslant k < S_n\} $
of the~length $Y_n:= S_n-S_{n-1}$ and for $\zeta^{(1)}_n \geqslant \zeta^{(2)}_n \geqslant \ldots \geqslant \zeta^{(r)}_n$ denoting first $r$ maxima in $C_n$ (we often write $\zeta_n$ for $\zeta^{(1)}_n$ and put $\zeta^{(q)}_n:=-\infty$ if $q > Y_n$), both of the following conditions hold:
\begin{enumerate}
\item $Y_n$ are independent for $n\geqslant 0$ and identically distributed for $n\geqslant 1 $;
\item $\left(\zeta_n, \zeta_n^{(2)}, \ldots, \zeta_n^{(r)}\right)$  are independent for $n\geqslant 0$ and identically distributed for $n\geqslant 1. $
\end{enumerate}
The process $\{X_n\}$ with an $r$-max-regenerative structure is called \emph{nondelayed} whenever $Y_0$ and $(\zeta_0, \zeta_0^{(2)}, \ldots, \zeta_0^{(r)})$ have the same distributions as $Y_1$ and $(\zeta_1, \zeta_1^{(2)}, \ldots, \zeta_1^{(r)})$, respectively, and \emph{delayed} otherwise. We denote by 
$\mu$ the~mean length of the cycle, i.e., $\mu:=\E Y_1.$
In our considerations $\mu$ is always finite.

It is natural to look for examples of regenerative processes in a class of Markov processes. Suppose that $\{X_n\}$ is a Markov process. We say that a Borel set $R\in\mathcal{B}(\mathbb{R})$ is \emph{recurrent} if  for every $x\in\mathbb{R}$ condition $\P( \inf\{n>0:X_n\in R\} <\infty \, | \, X_0=x)=1$ holds. A~set $R\in\mathcal{B}(\mathbb{R})$ is called a~\emph{regeneration set} if it is recurrent and, moreover, for some $j_0\geqslant 1$, some $\varepsilon\in(0,1)$ and some probability measure $\lambda$ on $(\mathbb{R}, \mathcal{B}(\mathbb{R}))$, for all $x\in R$ and $A\in \mathcal{B}(\mathbb{R})$, the~inequality
$\P(X_{j_0}\in A \, | \, X_0 =x)\geqslant \varepsilon \cdot \lambda(A)$
is true. If $\{X_n\}$ has a regeneration set, we say that it is \emph{Harris recurrent}.
It is shown \cite[Section VII.3]{AS}, that if $\{X_n\}$ is Harris recurrent with $j_0=1$, then one can construct this process simultaneously with a renewal process $\{S_n\}$ which makes $\{X_n\}$ regenerative and then $\{X_n\}$ has an $r$-max-regenerative structure for all $r\in\mathbb{N}_+$. In~the case $R=\{x_0\}$, one can put $S_{n}:=\inf\{k>S_{n-1} : X_k=x_0\}$ for $n\in\mathbb{N}$, with $S_{-1}=0$.

Rootz{\'e}n \cite[Theorem 3.1]{RO} describes the asymptotics of first maxima for~regenerative processes satisfying condition
\begin{equation}\label{C_0}
\P\left(\zeta_0 > \max_{1\leqslant k\leqslant n} \zeta_k\right) \to 0,\quad \text{as}\quad n\to\infty,
\end{equation}
in the following way.
\begin {thm}\label{TH_RO}
Suppose that $\{X_n\}$ has a $1$-max-regenerative structure with $\mu<\infty$ and, moreover, (\ref{C_0}) holds. Let $G(x):=\P(\zeta_1 \leqslant x)^{1 / \mu}$ for $x\in\mathbb{R}$. Then
\begin{equation}\label{PDF}
\sup_{x\in\mathbb{R}} \left|\P(M_n\leqslant x)  -  G(x)^n\right| \to 0\quad\text{as} \quad n\to\infty.
\end{equation}
\end {thm}

If (\ref{PDF}) holds with a~distribution function $G$, then we call $G$ a~\emph{phantom distribution function} for $\{X_n\}$, following O'Brien \cite{OBRIEN}. We will write $G_*$ for the right endpoint of $G$, i.e., $G_*:=\sup\{x\in\mathbb{R}:G(x) <1\}$. 

We note that assumption (\ref{C_0}) ensures that $\zeta_0$, with an arbitrary distribution in general, is negligible. Some examples of processes satisfying this condition are given in Remark~\ref{C_0_EX}.

Our goal is to describe the asymptotic behaviour of $q$-th maxima for regenerative processes. In order to do this, we combine the proof of Theorem \ref{TH_RO} by Rootz\'en and the methods applied for stationary sequences by Hsing \cite{HSI} (see also: Alpuim \cite{ALP}, Jakubowski \cite{JAK93}). In~Section~\ref{SECOND} we establish Theorem~\ref{SEC_MAX} and Corollary~\ref{WN} describing the~asymptotics of~$M_n^{(q)}$. Examples illustrating these results can be found in Section~\ref{THIRD}. We note that in the case, when the limiting cluster size distribution exists (i.e., for $\beta_i(x)$ defined in Section~\ref{SECOND}, $\beta_i(x)\to\beta_i$, as $x\nearrow G_*$, and $\sum_{i=1}^{\infty}\beta_i =1$ hold), Corollary \ref{WN} is a straightforward consequence of well known results on the exceedance point process; see, e.g., Lindvall \cite[Theorem 2]{LIND}, Rootz{\'e}n \cite[Theorem 3.3]{RO}, Asmussen \cite[Theorem 3.2]{ASS}. In~our approach, we do not involve the theory of point processes and do not assume that the size of the cluster of high threshold exceedances converges in distribution.

\section{Results}\label{SECOND}
We investigate $\{X_n\}$ with an $r$-max-regenerative structure and satisfying~(\ref{C_0}). 
The main result is the following theorem giving an asymptotic representation for $M_n^{(q)}$. 

\begin {thm}\label{SEC_MAX}
Suppose that $\{X_n\}$ has an $r$-max-regenerative structure with $\mu<\infty$ and satisfies~(\ref{C_0}).
Let $G(x):=\P(\zeta_1 \leqslant x)^{1/\mu}$ and $\beta_i(x):= \P(\zeta_1^{(i+1)} \leqslant x < \zeta_1^{(i)} \, | \, \zeta_1 > x)$ for $x\in\mathbb{R}$ and $i\in\{1,2,\ldots,r-1\}$.
Then, for every $q\in\{2,3,\ldots,r\}$,
\begin{equation*}
\sup_{x\in\mathbb{R}} \left|\P\left(M^{(q)}_n\leqslant x\right)  -  G(x)^n  \sum_{k=0}^{q-1}\frac{\left(-\log G(x)^n
\right)^k}{k!}\gamma_{q,k}(x)\right| \to 0, \quad \text{as}\quad n\to\infty,
\end{equation*}
holds with $\gamma_{q,k}(x)\in[0,1]$ given by
\begin{equation*}
\gamma_{q,k}(x)
:=\sum_{(j_1,j_2,\ldots, j_{q-1})\in J_{q,k}}
\frac{k!}{j_1! j_2! \cdots j_{q-1}!}\beta_1(x)^{j_1}\beta_2(x)^{j_2}\cdots \beta_{q-1}(x)^{j_{q-1}},
\end{equation*}
where $J_{q,k}:=\{(j_1,\ldots,j_{q-1})\in\mathbb{N}^{q-1}: \sum_{i=1}^{q-1} j_i=k,\sum_{i=1}^{q-1}ij_i\leqslant q-1\}$.

\end {thm}

\begin{proof}
Observe that it is sufficient to show that
\begin{eqnarray*}
\P\left(M^{(q)}_n\leqslant x_n\right)  -  G(x_n)^n \sum_{k=0}^{q-1}\frac{\left(-\log G(x_n)^n
\right)^k}{k!}\gamma_{q,k}(x_n) \to 0, \quad \text{as}\quad n\to\infty,
\end{eqnarray*}
for an arbitrary $\{x_n\}\subset\mathbb{R}$ satisfying $G(x_n)^n\to\alpha $ with some $\alpha \in[0,1]$. Let $\{x_n\}$ be such a~sequence.

We present a detailed proof for $r=2$. In this case we shall prove that
\begin{equation}\label{SUFF}
\P\left(M^{(2)}_n\leqslant x_n\right)  -  G(x_n)^n (1-\beta_1(x_n) \log G(x_n)^n) \to 0,\quad\text{as} \quad n\to\infty.
\end{equation}
Using arguments similar to those of \cite{RO}, for $\nu_n$ defined as 
$\nu_n:=\inf\{k\in\mathbb{N}: S_k>n\}$
and for an~arbitrary $\delta\in (0,1/\mu)$, we obtain
\begin{multline}\label{E1}
\P\left(M^{(2)}_{S_{\lfloor n/\mu + n\delta \rfloor} }\leqslant x_n\right) - \P\left(\left|\frac{\nu_n + 1}{n} - \frac{1}{\mu}\right| > \delta\right) \\
\leqslant\P\left(M_{S_{\nu_n}}^{(2)}\leqslant x_n\right)\leqslant \P\left(M_n^{(2)}\leqslant x_n\right) \leqslant \P\left(M_{S_{\nu_n -1}}^{(2)}\leqslant x_n,\; \nu_n>0\right) + \P(\nu_n = 0)\\
\leqslant \P\left(M^{(2)}_{S_{\lfloor n/\mu - n\delta \rfloor} }\leqslant x_n\right) + \P\left(\left|\frac{\nu_n-1}{n} - \frac{1}{\mu}\right| > \delta\right) + \P(\nu_n = 0),
\end{multline}
where $\P(\nu_n = 0)\to 0$ and  $\P(|\nu_n/n - 1/\mu| > \delta)\to 0$ due to the law of large numbers. Obviously, we also have
$$\P\left(M^{(2)}_{S_{\lfloor n/\mu + n\delta \rfloor} }\leqslant x_n\right) \leqslant \P\left(M^{(2)}_{S_{\lfloor n/\mu \rfloor} }\leqslant x_n\right) \leqslant \P\left(M^{(2)}_{S_{\lfloor n/\mu - n\delta \rfloor} }\leqslant x_n\right)$$
and, moreover,
\begin{eqnarray*}
\lefteqn{\left|\P\left(M^{(2)}_{S_{\lfloor n/\mu - n\delta \rfloor} }\leqslant x_n\right) - \P\left(M^{(2)}_{S_{\lfloor n/\mu +  n\delta \rfloor} }\leqslant x_n\right)\right|} \\
& \leqslant & \P\left(\zeta_k>x_n \text { for some } \lfloor n/\mu - n\delta \rfloor < k \leqslant \lfloor n/\mu + n\delta \rfloor +1 \right)\\
&=& 1 - \P(\zeta_1 \leqslant x_n)^{2n\delta} + o(1) \; = \;  1-G(x_n)^{2n \mu \delta} +o(1) \; = \; 1 - \alpha ^{2\mu\delta}+o(1).
\end{eqnarray*}

First, we prove (\ref{SUFF}) for $\alpha \in(0,1]$. From the above considerations and the~fact that $1 - \alpha ^{2\mu\delta} \to 0$, as $\delta\to 0$, we conclude:
\begin{equation*}
\P\left(M_n^{(2)}\leqslant x_n\right)=\P\left(M_{S_{\lfloor n/\mu \rfloor}}^{(2)}\leqslant x_n\right)+o(1).
\end{equation*}
Observe that we also have
\begin{eqnarray*}
\lefteqn{\left|\P\left(M_{S_{\lfloor n/\mu \rfloor}}^{(2)}\leqslant x_n\right) - \P\left(M^{(2)}_{S_0-1,S_{\lfloor n/\mu \rfloor}-1}\leqslant x_n \right)\right|}\\
&\leqslant & \P\left(\zeta_0>\max_{1\leqslant k\leqslant  \lfloor n/\mu \rfloor}\zeta^{(2)}_k\right) + \P\left(X_{S_{\lfloor n/\mu \rfloor}} > x_n\right)\\
&\leqslant &\P\left(\zeta_0>\max_{1\leqslant 2k\leqslant  \lfloor n/\mu \rfloor}\zeta_{2k}\right)+ \P\left(\zeta_0>\max_{1\leqslant 2k+1\leqslant  \lfloor n/\mu \rfloor}\zeta_{2k+1}\right) + \P\left(X_{S_{\lfloor n/\mu \rfloor}} > x_n\right),
\end{eqnarray*}
which combined with condition (\ref{C_0}) and the convergence $\P(X_{S_{\lfloor n/\mu \rfloor}} > x_n)\to 0$,  implies
\begin{equation}\label{APPROX2}
\P\left(M_{S_{\lfloor n/\mu \rfloor}}^{(2)}\leqslant x_n\right) = \P\left(M^{(2)}_{S_0-1,S_{\lfloor n/\mu \rfloor}-1}\leqslant x_n \right) + o(1).
\end{equation}
Next, note that
\begin{multline*}
\P\left(M^{(2)}_{S_0-1,S_{\lfloor n/\mu \rfloor}-1}\leqslant x_n \right) \\
= \P\left(M_{S_0-1,S_{\lfloor n/\mu \rfloor}-1}\leqslant x_n \right) + \P\left(M^{(2)}_{S_0-1,S_{\lfloor n/\mu \rfloor}-1}\leqslant x_n <M_{S_0-1,S_{\lfloor n/\mu \rfloor}-1} \right).
\end{multline*}
We can approximate the first summand of the right-hand side of the above equality as follows:
\begin{equation*}
\P\left(M_{S_0-1,S_{\lfloor n/\mu \rfloor}-1}\leqslant x_n \right) 
= \P(\zeta_1\leqslant x_n)^{\lfloor n/\mu \rfloor}= G(x_n)^n + o(1).
\end{equation*}
For the second one we obtain that
\begin{eqnarray*}
\lefteqn{\P\left(M^{(2)}_{S_0-1,S_{\lfloor n/\mu \rfloor}-1}\leqslant x_n <M_{S_0-1,S_{\lfloor n/\mu \rfloor}-1} \right)}\\
& =& \sum_{k=1}^{\lfloor n/\mu \rfloor} \P\left(\zeta^{(2)}_k\leqslant x_n < \zeta_k,\, \zeta_i\leqslant x_n \text{ for }i\in\{1,2,\ldots, \lfloor n/\mu \rfloor\} \backslash \{k\}\right)\\
& =& \lfloor n/\mu\rfloor \P(\zeta_1> x_n) \P\left(\zeta_1^{(2)}\leqslant x_n\, \Big| \, \zeta_1 >x_n\right)\P(\zeta_1 \leqslant x_n)^{ \lfloor n/\mu\rfloor - 1}\\
& =& n/\mu \cdot (1-G(x_n)^\mu) \beta_1(x_n) G(x_n)^n + o(1)
\; = \; - G(x_n)^n \beta_1(x_n) \log G(x_n)^n + o(1).
\end{eqnarray*}
The~convergence (\ref{SUFF}) for $\alpha \in(0,1]$ follows.

To finish the proof in the case $r=2$, we need to show that (\ref{SUFF}) holds when $\alpha =0$. Since $\beta_1(x_n)\in[0,1]$, we easily get that $G(x_n)^n (1-\beta_1(x_n) \log G(x_n)^n) \to 0$.
It is sufficient to prove that also $\P(M_n^{(2)}\leqslant x_n)\to 0$. Applying (\ref{E1}) with $\delta:=(2\mu)^{-1}$, we obtain
$$ \P\left(M_n^{(2)}\leqslant x_n\right)  \leqslant \P\left(M^{(2)}_{S_{\lfloor n/(2\mu) \rfloor} }\leqslant x_n\right) + \P\left(\left|\frac{\nu_n-1}{n} - \frac{1}{\mu}\right| > \frac{1}{2\mu}\right) + \P(\nu_n=0).$$
Since both the second and the third summand of the right-hand side tend to zero, we conclude that
\begin{eqnarray*}
\P\left(M_n^{(2)}\leqslant x_n\right) & \leqslant &\P\left(M^{(2)}_{S_{\lfloor n/(2\mu) \rfloor} } \leqslant x_n\right) + o(1) \\
&\leqslant & \P(\zeta_k \leqslant x_n \textrm{ for all odd }k\in\{1,2,\ldots, \lfloor n/(2\mu) \rfloor \}) \\
&& \; +\P(\zeta_k \leqslant x_n \textrm{ for all even }k\in\{1,2,\ldots, \lfloor n/(2\mu) \rfloor \}) + o(1)\\
&=& 2G(x_n)^{n/4}+o(1) \; = \; o(1),
\end{eqnarray*}
which completes the proof for $r=2$.

In the case $r>2$, using arguments similar to the ones presented above, we show that
\begin{equation*}
\P\left(M_n^{(q)}\leqslant x_n\right)=\P\left(M_{S_{\lfloor n/\mu \rfloor}}^{(q)}\leqslant x_n\right)+o(1)=\P\left(M^{(q)}_{S_0-1,S_{\lfloor n/\mu \rfloor}-1}\leqslant x_n \right)+o(1)
\end{equation*}
holds for each $q\in\{2,3,\ldots,r\}$. Then, following, e.g., Hsing \cite[Corollary 3.2]{HSI}, we get that
\begin{multline*}
\P\left(M^{(q)}_{S_0-1,S_{\lfloor n/\mu \rfloor}-1}\leqslant x_n \right)\\
= \sum_{k=0}^{q-1} \P\left(M^{(q)}_{S_0-1,S_{\lfloor n/\mu \rfloor}-1}\leqslant x_n, \#\{l\in\{1,2,\ldots,\lfloor n/\mu \rfloor\}: \zeta_l>x_n\} = k \right),
\end{multline*}
with $k$ denoting the number of cycles with maxima exceeding $x_n$,
and, furthermore,
\begin{eqnarray*}
\lefteqn{\P\left(M^{(q)}_{S_0-1,S_{\lfloor n/\mu \rfloor}-1}\leqslant x_n, \#\{l\in\{1,2,\ldots,\lfloor n/\mu \rfloor\}: \zeta_l>x_n\} = k \right)}\\
&=& \!\!\!\!\!\! \sum _{(j_1,j_2, \ldots, j_{q-1}) \in J_{q,k}} \! \frac{\lfloor n/\mu \rfloor !}{ j_1! j_2! \cdots j_{q-1}!(\lfloor n/\mu \rfloor \! - \! k)!} \P(\zeta_1 \leqslant x_n)^{\lfloor n/\mu \rfloor \! - \!k} \prod_{i=1}^{q-1}\P\left(\zeta_1^{(i+1)} \leqslant x_n <\zeta_1^{(i)}\right)^{j_i} ,
\end{eqnarray*}
with $j_i$ denoting the number of cycles with exactly $i$ exceedances. To complete the proof of the theorem, it is sufficient to notice that
\begin{eqnarray*}
\lefteqn{\frac{\lfloor n/\mu \rfloor !}{ j_1! \cdots j_{q-1}!(\lfloor n/\mu \rfloor \!\! - \!\! k)!} \P(\zeta_1 \! \leqslant \! x_n)^{\lfloor n/\mu \rfloor \! - \! k}\P\left(\zeta_1^{(2)} \! \leqslant \! x_n \! < \! \zeta_1\right)^{j_1} \cdots \P\left(\zeta_1^{(q)} \! \leqslant \! x_n \! < \! \zeta_1^{(q\!-\!1)}\right)^{j_{q-1}}}\\
&=& \frac{1}{ j_1! \cdots j_{q-1}!}\lfloor n/\mu\rfloor^k \P(\zeta_1 \! \leqslant \! x_n)^{\lfloor n/\mu \rfloor}(1-\P(\zeta_1 \!\leqslant \! x_n))^k \beta_1(x_n)^{j_1}\cdots \beta_{q-1}(x_n)^{j_{q-1}} \! + \! o(1)\\
&=& \frac{1}{ j_1! \cdots j_{q-1}!} G(x_n)^n (-\log G(x_n)^n)^k \beta_1(x_n)^{j_1}\cdots \beta_{q-1}(x_n)^{j_{q-1}} \! + \! o(1)
\end{eqnarray*}
holds.
\end{proof}

\begin{cor}\label{WN}
Let the assumptions of Theorem \ref{SEC_MAX} be satisfied and let 
$$\beta_i(x)\to \beta_i\quad \text{as} \quad x \nearrow G_*,\quad\text{for some}\quad \beta_i\in[0,1], \quad\text{for all}\quad i\in\{1,2,\ldots, r-1\}.$$
Then
\begin{equation}\label{LIM}
\sup_{x\in\mathbb{R}} \left|\P\left(M^{(q)}_n\leqslant x\right)  -  G(x)^n  \sum_{k=0}^{q-1}\frac{\left(-\log G(x)^n
\right)^k}{k!}\gamma_{q,k}\right| \to 0, \quad \text{as}\quad n\to\infty,
\end{equation}
holds with
$
\gamma_{q,k}
:=\sum_{(j_1,j_2,\ldots, j_{q-1})\in J_{q,k}}
\frac{k!}{j_1! j_2! \cdots j_{q-1}!}\beta_1^{j_1}\beta_2^{j_2}\cdots \beta_{q-1}^{j_{q-1}}.
$
\end{cor}

\begin{rmk}\label{I_EX}
It is quite easy to define a regenerative process $\{X_n\}$ with $\beta_i(\cdot)$ convergent only for $i\in I$, where $I\subset\mathbb{N}_+$ is fixed in advance.

\begin{enumerate}

\item[a)] Let $p(m,i)\geqslant 0$ satisfy $\sum_{i\in\mathbb{N}_+}p(m,i)=1$ for every $m\in\mathbb{N}_+$. Let $\{(V_k,Y_k):k\in\mathbb{N}\}$ be an i.i.d. sequence with the distribution of $(V_1,Y_1)$ given by $\P(V_1=m)=2^{-m}$ and $\P(Y_1=i\,|\, V_1=m)=p(m,i)$ for all $m,i\in\mathbb{N}_+$. Define $X_n:=V_{k(n)}$ with $k(n)$ such that $\sum_{j=0}^{k(n)-1}Y_j \leqslant n <\sum_{j=0}^{k(n)}Y_j$. Then $\{X_n\}$ is regenerative and we have $\beta_i(m)=2^m\sum_{v=m+1}^\infty 2^{-v}p(v,i)$. It is clear that $p(m,i)\to \beta_i$ implies $\beta_i(m)\to \beta_i$, as~$m\to\infty$. Since $\beta_i(m)=(\beta_i(m+1) + p(m+1,i))/2$, we conclude that $p(m,i)\to \beta_i$ follows from $\beta_i(m)\to \beta_i$.

\item[b)] Let $p(x,i)\geqslant 0$ be such that $\sum_{i\in\mathbb{N}_+}p(x,i)=1$ holds for every $x>0$ and let the function $\phi_{i}(x):=p(x,i)$ be uniformly continuous for every $i\in\mathbb{N}_+$.  Let $\{(V_k, Y_k):k\in\mathbb{N}\}$ be an~i.i.d. sequence with the distribution of $(V_1,Y_1)$ given by $\P(V_1>x)=e^{-x}$ and $\P(Y_1=i\,|\, V_1=x)=p(x,i)$ for all $x>0$ and $i\in\mathbb{N}_+$. Consider $\{X_n\}$ defined as above. Then $\{X_n\}$ is regenerative and $\beta_i(x)=e^x\int_{x}^\infty e^{-v}\phi_i(v) dv$. It is easy to show that $\phi_{i}(x)\to \beta_i$ implies $\beta_i(x)\to\beta_i$, as $x\to\infty$. Moreover, since 
$\beta_i(x)=e^{-\varepsilon}\beta_i(x+\varepsilon) + \int_{x}^{x+\varepsilon} e^{-(v-x)}\phi_{i}(v) dv$, for all $\varepsilon > 0$,
and $\phi_i$ is uniformly continuous, we obtain that $\phi_{i}(x)\to \beta_i$ is a consequence of $\beta_i(x)\to\beta_i$.

\end{enumerate}

\end{rmk}

\begin{rmk}\label{C_0_EX}
Condition (\ref{C_0}) guarantees that the cycle $C_0$ does not affect the asymptotic behaviour of $M^{(q)}_n$, for any $q\in\mathbb{N}$; see, e.g., the argumentation for (\ref{APPROX2}) in the proof of Theorem~\ref{SEC_MAX}. 
\begin{enumerate}
\item[a)] If $\{X_n\}$ with a 1-max-regenerative structure is nondelayed, then it satisfies (\ref{C_0}).
\item[b)] The process $\{X_n\}$, which is  stationary and regenerative in the sense of Asmussen \cite{AS}, with $\mu<\infty$, fulfills (\ref{C_0}); it follows from \cite[Corollary VI.1.5]{AS}.
\item[c)] $\{X_n\}$ with a $1$-max-regenerative structure satisfies (\ref{C_0}) if and only if \mbox{$x_*^{\zeta_0}\leqslant x_*^{\zeta_1}$} holds and $\P(\zeta_0=x_*^{\zeta_1})>0$ implies $\P(\zeta_1=x_*^{\zeta_1})>0$, where $x_*^{\zeta_i}:=\sup\{x\in\mathbb{R}:\P(\zeta_i\leqslant x) <1\}$ for $i\in\{0,1\}$.
\end{enumerate}
\end{rmk}

\section{Examples}\label{THIRD}
In this section we give a few examples of nondelayed regenerative Markov processes and apply Theorem~\ref{TH_RO} and Corollary~\ref{WN} to describe their phantom distribution functions~$G$ and to calculate the constants $\beta_1, \beta_2, \ldots$ such that condition (\ref{LIM}) holds.

In our considerations, we make use of the stopping moments:
\begin{eqnarray*}
\tau_a  &:= & \inf\{k>0: X_k=a\},\\
\tau_{a^+} &:= & \inf\{k>0: X_k>a\},\\
\tau_{a,b} &:= & \inf\{k>0: X_l = a  \text{ and } X_k= b\text{ for some } 0<l<k\},
\end{eqnarray*}
defined for $a,b\in\mathbb{R}$.

In some of the presented examples,
a \emph{long-tailed} distribution function $F$, i.e., satisfying condition 
\begin{equation}\label{EL}
\lim_{x\to\infty} \frac{1-F(x+y)}{1-F(x)}= 1,\quad\text{for all}\quad y>0,
\end{equation}
plays a crucial role. We recall that every distribution function~$F$ which is \emph{subexponential}
is also long-tailed \cite[Lemma 1.3.5]{KLUP}.

Sometimes, it is convenient to describe the~phantom distribution function $G$ giving a~distribution function $\tilde{G}$ such that the following convergence
$$\sup_{x\in\mathbb{R}} \left| G(x)^n - \tilde{G}(x)^n\right|\to 0,\quad \text{as} \quad n\to\infty,$$
holds. If $G$ and $\tilde{G}$ satisfy the above condition, we call them \emph{strictly tail-equivalent}.

In two examples, we investigate the \emph{Lindley process} defined recursively as 
\begin{equation}\label{DEF_LINDLEY}
X_n := \left\{\begin{array}{ll}
0 &\text{for}\quad n=0;\\
\max\{X_{n-1}+Z_{n-1},0\} &\text{for}\quad n\in\mathbb{N}_+,
\end{array}\right.
\end{equation}
for $\{Z_n:n\in\mathbb{N}\}$ a sequence of i.i.d. random variables with  $\E Z_0 <0$. Note that such a process is a~nondelayed regenerative Markov process with the regeneration set  $R=\{0\}$.

\subsection{Geometric jump}

Let $\{X_n\}$ be a Markov chain with the state space $\mathbb{N}$ such that $X_0=0$ and the transition probabilities are given by 
$\P(0,k)=p(1-p)^k$ for $k\in\mathbb{N}$ and $\P(k,k-1)=1$ for $k\in\mathbb{N}_+$, for some $p\in(0,1)$. Then  $\{X_n\}$ is regenerative and nondelayed with the~regeneration set $R=\{0\}$ and $\mu=p^{-1}$. The phantom distribution function $G$ is given by
$$G(x)=\P(\zeta_0\leqslant x)^{1/\mu} = \P(X_1\leqslant \lfloor x \rfloor)^p = \left (1-(1-p)^{\lfloor x \rfloor +1}\right)^p,\quad\text{for} \quad x>0,$$
and we calculate $\beta_i$ for any $i\geqslant 1$ as follows:
\begin{equation*}
\beta_i=\lim_{x\to\infty}\P\left(\zeta_0^{(i+1)}\leqslant x < \zeta_0^{(i)} \,\Big|\, \zeta_0 > x\right)=\lim_{x\to\infty}\P\left(X_1=\lfloor x \rfloor + i \,\big|\, X_1 > \lfloor x \rfloor \right)=p(1-p)^{i-1}.
\end{equation*}

\subsection{Reflected simple random walk}
Let $\{Z_n\}$ be an i.i.d. sequence such that \mbox{$\P(Z_0=1)=p$} and \mbox{$\P(Z_0=-1)=q$} with $q:=1-p$, for some \mbox{$p<1/2$}.
Let $\{X_n\}$ be the process defined by (\ref{DEF_LINDLEY}).
Then 
$$\mu = \E\tau_0\mathbb{1}_{\{X_1=0\}}+\E\tau_0\mathbb{1}_{\{X_1=1\}} = q + p(1+(q-p)^{-1})=  q/(q-p),$$
where $\E(\tau_0-1\, | \, X_1=1)=(q-p)^{-1}$ follows from Wald's identity. Applying the optional stopping theorem to the classical gambler's ruin problem, we get
$$\P(\tau_0<\tau_{a+b}\,|\,X_0=a)=\frac{(q/p)^{a+b}-(q/p)^a}{(q/p)^{a+b}-1},\quad\text{for all}\quad a,b\in\mathbb{N}_+.$$
Hence, the phantom distribution function~$G$ is given by
\begin{eqnarray*}
G(x)&=&\P(\zeta_0\leqslant x)^{1/\mu}\\
&=&\left(\P(X_1=0\,|\,X_0=0)+\P\left(X_1=1, \tau_0<\tau_{\lfloor x\rfloor+1}\,|\, X_0=0\right)\right)^{1/\mu}\\
&=& \left(q + p \frac{(q/p)^{\lfloor x\rfloor +1}-(q/p)}{(q/p)^{\lfloor x\rfloor+1}-1}\right)^{1 - p/q},
\end{eqnarray*}
for $x>0$, and we can calculate the constant $\beta_1$ as follows:
\begin{eqnarray*}
\beta_1 &=& \lim_{x\to\infty} \P\left(\zeta_0^{(2)} \leqslant x \,\Big| \, \zeta_0 > x\right) \\
&=& \lim_{x\to\infty} \P\left(\tau_{\lfloor x\rfloor+1,0} < \tau_{\lfloor x\rfloor+1,\lfloor x\rfloor+1}, X_{\tau_{\lfloor x\rfloor+1}+1}=\lfloor x \rfloor \,\Big|\,  \tau_{\lfloor x\rfloor+1}<\tau_0 \right)\\
&=&\lim_{x\to\infty} q\frac{(q/p)^{\lfloor x\rfloor+1}-(q/p)^{\lfloor x\rfloor }}{(q/p)^{\lfloor x\rfloor+1}-1} \; = \; q-p.
\end{eqnarray*}
One can obtain $\beta_i$ for $i\geqslant 2$ in a similar (but more complicated) way. An interesting analysis of the reflected simple random walk has been done by Lindvall \cite[Section 3.B]{LIND}.

\subsection{Lindley process with long-tailed steps}
Let $\{Z_n\}$ be a sequence of i.i.d. random variables with a distribution function $F$ satisfying condition~(\ref{EL}) and such that $\E Z_1<0$.
Let $\{X_n\}$ be the process defined by (\ref{DEF_LINDLEY}) and suppose that $\mu<\infty$ holds.
It is known \cite[Theorem 2.1]{ASS} that then
\begin{equation}\label{ASS_OGON}
\frac{\P(\zeta_0 > x)}{\mu(1-F(x))}\to 1 \quad \text{as} \quad x\to\infty.
\end{equation}
Thus the distribution functions $G(x)=\P(\zeta_0\leqslant x)^{1/\mu}$ and $F(x)$ are strictly tail-equivalent.
We will show that $\beta_i=0$ for every $i\in\mathbb{N}_+$. For a fixed $i\in\mathbb{N}_+$ and an arbitrary $\varepsilon > 0$, let $y>0$ be chosen so that $F(-y/i)<\varepsilon/i$. Then
\begin{eqnarray*}
\beta_i(x)&\leqslant &\P\left(\zeta_0^{(i+1)}\leqslant x \,\Big|\, \zeta_0>x\right)=\frac{\P\left(\zeta^{(i+1)}_0 \leqslant x, \zeta_0 >x\right)}{\P(\zeta_0>x+y)} \cdot \frac{\P(\zeta_0>x+y)}{\P(\zeta_0>x)}\\
&=& \frac{\P\left(\zeta^{(i+1)}_0 \leqslant x, \zeta_0 >x+y\right) + \P\left(\zeta^{(i+1)}_0 \leqslant x, \zeta_0 \in(x,x+y]\right)}{\P(\zeta_0>x+y)}  \cdot \frac{\P(\zeta_0>x+y)}{\P(\zeta_0>x)}\\
&\leqslant& \left(\P\left(\zeta^{(i+1)}_0 \leqslant x \,\Big|\, \zeta_0>x+y\right) + \frac{\P(\zeta_0 \in(x,x+y])}{\P(\zeta_0>x+y)}\right)\cdot \frac{\P(\zeta_0>x+y)}{\P(\zeta_0>x)}.
\end{eqnarray*}
Combining (\ref{EL}) and (\ref{ASS_OGON}), we conclude that
$$\frac{\P(\zeta_0 \in(x,x+y])}{\P(\zeta_0>x+y)}\to 0\quad \text{and} \quad \frac{\P(\zeta_0>x+y)}{\P(\zeta_0>x)}\to 1,\quad \text{as} \quad x\to\infty.$$
Moreover, we have
\begin{eqnarray*}
\P\left(\zeta^{(i+1)}_0 \leqslant x \,\bigg|\, \zeta_0>x+y\right) 
&\leqslant &\P\left(\min_{1\leqslant j \leqslant i}X_{j+\tau_{(x+y)^+}} \leqslant x  \,\bigg|\, \tau_{(x+y)^+} < \tau_0\right)\\
&\leqslant & \P\left(\min_{0\leqslant j \leqslant i-1} Z_{j+\tau_{(x+y)^+}} \leqslant -\frac{y}{i}  \,\bigg|\, \tau_{(x+y)^+} < \tau_0\right)\\
&\leqslant & i F\left(-\frac{y}{i}\right) \; < \; \varepsilon.
\end{eqnarray*}
Since $\varepsilon>0$ is arbitrarily small, $\beta_i(x)\to 0$ follows and hence $\beta_i=0$.

\subsection{Regenerative Markov chain with $\beta_1, \beta_2, \ldots$ given in advance}\label{EX_CON}
Let $\{\beta_i : i\in\mathbb{N_+}\}\subset [0,1]$ be a sequence satisfying the following assumptions:
\begin{equation*}
\sum_{i=1}^\infty\beta_i =1  \quad \text{and} \quad \sum_{i=1}^\infty i \beta_i  < \infty.
\end{equation*}
For such $\{\beta_i\}$, it is easy to find a regenerative (non-Markovian) process with $\mu<\infty$, so that (\ref{LIM}) holds (see, e.g., Remark \ref{I_EX} above or Remark 3.3 in \cite{ALP}). Below, we construct a regenerative Markov chain satisfying condition~(\ref{LIM}) with $\beta_1, \beta_2, \ldots$ as above, given in advance.

Let $F$ be an~arbitrary long-tailed distribution function (see (\ref{EL})). Put $i_0:=\min\{i\in\mathbb{N}_+ : \beta_i>0\}$ and choose a sequence $\{m_n\}\subset\{i_0, i_0+1, \ldots\}$ to be nondecreasing and such that
\begin{equation}\label{M_N}
\frac{1-F(n+m_n)}{1-F(n)}\to 1 \quad \text{as} \quad n\to\infty.
\end{equation}
We define recursively an increasing sequence $\{v_n\}\subset\mathbb{N}_+$ as
$$ v_n:=\left\{ 
\begin{array}{ll}
1 & ,\quad\text{for} \quad n=1;\\
v_{n-1}+m_{v_{n-1}} &,\quad\text{for} \quad n\geqslant 2,
\end{array} \right.$$
and put $n(x):=\min\{n\in\mathbb{N}_+ : v_n > x\}$ for $x\in\mathbb{R}$. Note that $v_{n(x)-1} \leqslant x < v_{n(x)}$ for all $x\geqslant 1$.

Let $\{X_n\}$ be a Markov chain with the state space $\mathbb{N}$, such that $X_0=0$ and the transition probabilities are given by:
\begin{itemize}
\item
$\P(0,0)=F(1)$ and $\P(0,v_n)=F(v_{n+1})-F(v_n)$, for $n\in\mathbb{N}_+$,
\item
$\P(v_n, 0)=\beta_1 / \sum_{j=1}^{m_{v_n}} \beta_j$ and $\P(v_n, v_n+i-1)= \beta_{i} / \sum_{j=1}^{m_{v_n}} \beta_j$, for $n\in\mathbb{N}_+$ and \mbox{$i\in\{2,3,\ldots , m_{v_n}\}$},
\item
$\P(v_n+1,0)= \P(v_n +i-1, v_n+i-2)=1$, for $n\in\mathbb{N}_+$ and  $i\in\{3,4,\ldots , m_{v_n}\}$.
\end{itemize}
Figure \ref{FIG} illustrates the dynamics of this process.
Note that $\{X_n\}$ is a nondelayed regenerative process with the regeneration set $R=\{0\}$. Moreover, since $\E (Y_0 \, | \, X_1=v_n) \leqslant m_{v_n}+1$,
for every $n\in\mathbb{N}_+$, and
\begin{equation*}
\E (Y_0 \, | \, X_1=v_n)=\sum_{i=1}^{m_{v_n}} (i+1) \frac{\beta_i}{\sum_{j=1}^{m_{v_n}} \beta_j}
= \frac{1}{\sum_{j=1}^{m_{v_n}} \beta_j} \cdot \sum_{i=1}^{\infty}( i+1)\beta_i 
\leqslant 4 \cdot \sum_{i=1}^{\infty} i \beta_i 
 < \infty,
\end{equation*}
for all sufficiently large $n\in\mathbb{N}_+$,
the condition $\mu = \E Y_0 < \infty$ holds.

\begin{figure}
\centering
\includegraphics[width=14cm]{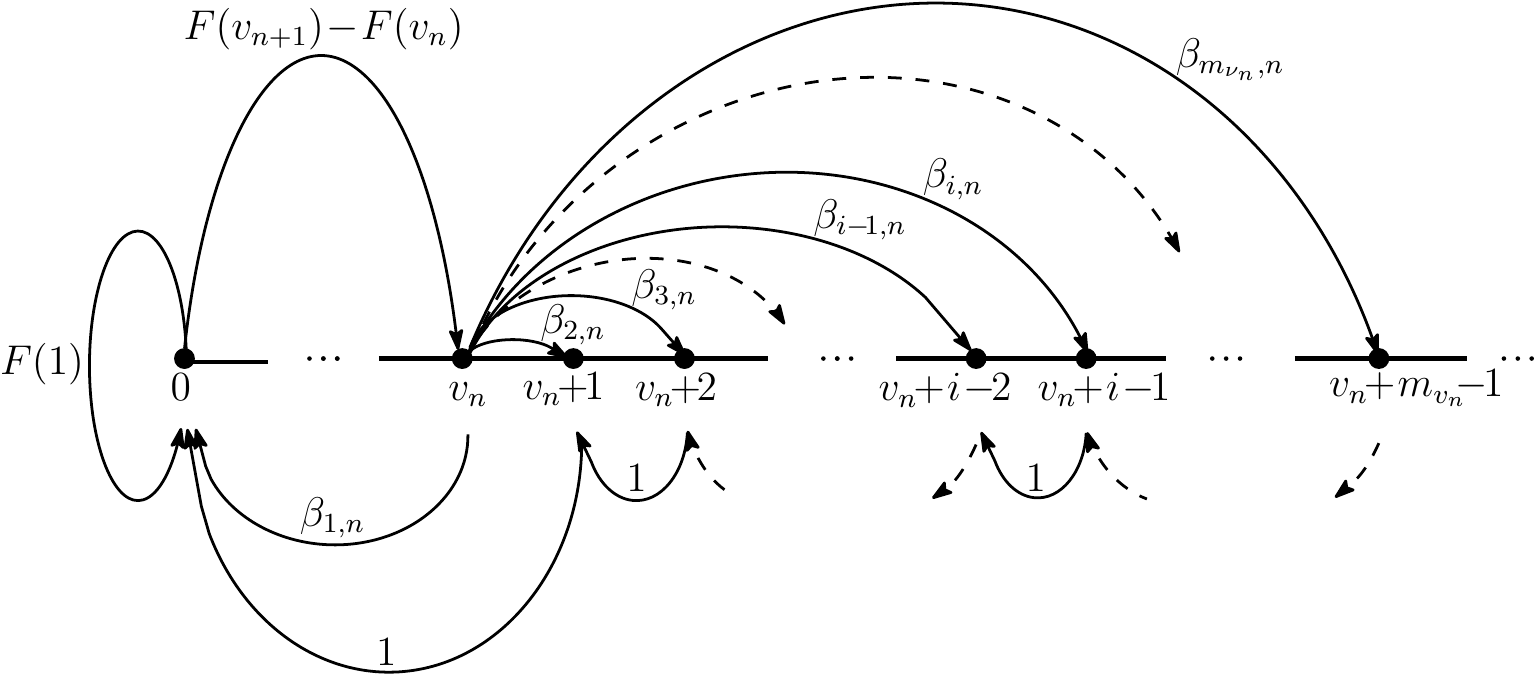}
\caption{The dynamics of the Markov chain from Example \ref{EX_CON}. Here, $n\in\mathbb{N}_+$ is arbitrary and \mbox{$\beta_{i,n} := \beta_i / \sum_{j=1}^{m_{v_n}} \beta_j$} for any $i\in\{1,2,\ldots, m_{v_n}\}$.}\label{FIG}
\end{figure}

In order to describe the phantom distribution function $G(x)=\P(\zeta_0\leqslant x)^{1/\mu}$, observe that 
$$1-F(v_{n(x)+1})=\P\left(X_1 > v_{n(x)}\right)\leqslant \P\left(\zeta_0 > x\right) \leqslant \P\left(X_1 > v_{n(x)-1}\right)=1-F\left(v_{n(x)}\right).$$
Thus we have
$$\frac{1-F(v_{n(x)+1})}{1-F(v_{n(x)-1})} \leqslant \frac{\P(\zeta_0 > x)}{1-F(x)} \leqslant \frac{1-F(v_{n(x)})}{1-F(v_{n(x)})} = 1, $$
which, combined with (\ref{M_N}), implies that
$$  \frac{\P\left(\zeta_0 > x\right)}{1-F(x)}  \to 1 \quad \text{as} \quad x\to\infty.$$
The strictly tail-equivalence of $G$ and $F^{1/\mu}$ follows.

We will show that $\{X_n\}$ satisfies (\ref{LIM}) with the sequence $\{\beta_i\}$.
Note that
\begin{eqnarray*}
\beta_i(x)&=&\P\left(\zeta_0^{(i+1)}\leqslant x < \zeta_0^{(i)} \, \Big| \, \zeta_0>x \right)\\
&=& \P\left(\zeta_0^{(i+1)}\leqslant x < \zeta_0^{(i)} \, \Big| \, \zeta_0 \geqslant v_{n(x)} \right) \cdot \P\left(\zeta_0 \geqslant v_{n(x)} \, \big| \, \zeta_0>x\right) \\
&& +   \P\left(\zeta_0^{(i+1)}\leqslant x < \zeta_0^{(i)} \, \Big| \,  v_{n(x)} > \zeta_0 > x \right) \cdot \P\left( v_{n(x)} > \zeta_0 \, \big| \, \zeta_0>x\right).
\end{eqnarray*}
Since $\zeta_0 \geqslant v_{n(x)}$ if, and only if, $X_1 \geqslant v_{n(x)}$, for all $x>v_2$ we obtain
\begin{eqnarray*}
\P\left(\zeta_0 \geqslant v_{n(x)} \, \big| \, \zeta_0>x\right) & \geqslant & \P\left(\zeta_0 \geqslant v_{n(x)} \, \big| \, \zeta_0 \geqslant v_{n(x)-1}\right) \\
&=&\P\left(X_1 \geqslant v_{n(x)} \, \big| \, X_1 \geqslant v_{n(x)-1}\right)=\frac{1-F(v_{n(x)-1})}{1-F(v_{n(x)-2})}.
\end{eqnarray*}
Combining this fact with the definition of $\{v_n\}$ and with condition (\ref{M_N}), we conclude that
$$\lim_{x\to\infty}\P\left(\zeta_0 \geqslant v_{n(x)} \, \big| \, \zeta_0>x\right) =1
\quad \text{and thus} \quad
\lim_{x\to\infty}\P\left( v_{n(x)} > \zeta_0  \, \big| \, \zeta_0>x\right) = 0.$$
Observe that we also have
\begin{eqnarray*}
\lefteqn{\P\left(\zeta_0^{(i+1)}\leqslant x < \zeta_0^{(i)} \, \Big| \, \zeta_0 \geqslant v_{n(x)} \right)}\\
&=& \sum_{k=n(x)}^\infty \P\left(\zeta_0^{(i+1)}\leqslant x < \zeta_0^{(i)} \, \Big| \, v_k \leqslant \zeta_0 < v_{k+1}\right)\cdot \P\left(v_k \leqslant \zeta_0 < v_{k+1} \,\Big|\, \zeta_0 \geqslant v_{n(x)} \right)\\
&=& \sum_{k=n(x)}^\infty \frac{\beta_i}{\sum_{j=1}^{m_{v_k}}\beta_j } \cdot \frac{F(v_{k+1}) - F(v_k)}{1 - F(v_{n(x)})}.
\end{eqnarray*}
Combining the above equality with the inequality
$$ 1 \leqslant \left(\sum_{j=1}^{m_{v_k}}\beta_j \right)^{-1} \leqslant \left(\sum_{j=1}^{m_{v_{n(x)}}}\beta_j\right)^{-1} \leqslant 1+\varepsilon,$$
true for all small $\varepsilon>0$, large $x$
and $k \geqslant n(x)$, keeping in mind that
$$\sum_{k=n(x)}^\infty (F(v_{k+1}) - F(v_k)) = 1 - F(v_{n(x)}),$$
we get
\begin{eqnarray*}
\beta_i \leqslant \lim_{x\to\infty} \sum_{k=n(x)}^\infty \frac{\beta_i}{\sum_{j=1}^{m_{v_k}}\beta_j} \cdot \frac{F(v_{k+1}) - F(v_k)}{1 - F(v_{n(x)})} \leqslant \beta_i (1+\varepsilon).
\end{eqnarray*}
Summarizing,
\begin{eqnarray*}
\lim_{x\to\infty}\beta_i(x)  & = & \lim_{x\to\infty} \P\left(\zeta_0^{(i+1)}\leqslant x < \zeta_0^{(i)} \, \Big| \, \zeta_0>v_{n(x)} \right) \\
&=& \lim_{x\to\infty} \sum_{k=n(x)}^\infty \frac{\beta_i}{\sum_{j=1}^{m_{v_k}}\beta_j} \cdot \frac{F(v_{k+1}) - F(v_k)}{1 - F(v_{n(x)})} = \beta_i.
\end{eqnarray*}

\section*{Acknowledgement}
The author would like to thank Adam Jakubowski for all suggestions and comments.


\begin{thebibliography}{9}

\bibitem{ALP} \textsc{Alpuim, M.T.}, High level exceedances in stationary sequences with extremal index, \textit{Stochastic Process. Appl. 30} (1988), 1--16.

\bibitem{ASS} \textsc{Asmussen, S.}, Subexponential asymptotics for stochastic processes: extremal behavior, stationary distributions and first passage probabilities, \textit{Ann. Appl. Probab. 8} (1998), 354--374.

\bibitem{AS} \textsc{Asmussen, S.}, \textit{Applied Probability and Queues}, Springer New York, 2003.

\bibitem{KLUP} \textsc{Embrechts, P., Kl{\"u}ppelberg, C., Mikosch, T.}, \textit{Modelling Extremal Events for Insurance and Finance}, Springer Berlin Heidelberg, 1997.

\bibitem{HSI} \textsc{Hsing, T.}, On the extreme order statistics for a stationary sequence, \textit{Stochastic Process. Appl. 29} (1988), 155--169.

\bibitem{JAK93} \textsc{Jakubowski, A.}, Asymptotic (r-1)-dependent representation for rth order statistic from a stationary sequence, \textit{Stochastic Process. Appl. 46} (1993), 29--46.

\bibitem{LIND} \textsc{Lindvall, T.}, An invariance principle for thinned random measures, \textit{Studia Sci. Math. Hung. 11} (1976), 269--275.

\bibitem{OBRIEN} \textsc{O'Brien, G.L.}, Extreme values for stationary and {M}arkov sequences, \textit{Ann. Probab. 15} (1987), 281--291.

\bibitem{RO} \textsc{Rootz{\'e}n, H.}, Maxima and exceedances of stationary {M}arkov chains, \textit{Adv. in Appl. Probab. 20} (1988), 371--390.

\end{thebibliography}
\end{document}